\documentclass[letterpaper, 10 pt, conference]{ieeeconf}  

\IEEEoverridecommandlockouts                              

\usepackage{amsmath} 
\usepackage{amssymb}  

\newtheorem{thm}{Theorem}
\newtheorem{lem}{Lemma}

\newtheorem{defn}{Definition}
\newtheorem{rem}{Remark}
\newtheorem{example}{Example}

\DeclareMathOperator{\Span}{span}
\DeclareMathOperator{\Rank}{rank}
\newtheorem{alg}{Algorithm}
\usepackage{comment}

\newcommand{\zcord}{\zeta}

\title{\LARGE \bf
A Normal Form for Two-Input Flat Nonlinear Discrete-Time Systems
}

\author{Johannes Diwold, Bernd Kolar and Markus Sch{\"o}berl
\thanks{The first author and the second author have been supported by the Austrian Science Fund (FWF) under grant number P~32151 and P~29964. All authors are with the Institute of Automatic Control and Control Systems Technology, Johannes Kepler University Linz, Altenbergerstrasse 66, 4040 Linz, Austria.\newline {Email: \tt\small johannes.diwold@jku.at,\newline bernd.kolar@jku.at, markus.schoeberl@jku.at}}
}

\begin{document}

\maketitle
\thispagestyle{empty}
\pagestyle{empty}

\begin{abstract}

We show that every flat nonlinear discrete-time system with two inputs can be transformed into a structurally flat normal form by state- and input transformations. This normal form has a triangular structure and allows to read off the flat output, as well as a systematic construction of the parameterization of all system variables by the flat output and its forward-shifts. For flat continuous-time systems no comparable normal form exist.
\end{abstract}

	



\setlength{\arraycolsep}{2pt} 

\section{Introduction}

In the 1990s, the concept of flatness has been introduced by Fliess,
L\'{e}vine, Martin and Rouchon for nonlinear continuous-time systems
(see e.g. \cite{FliessLevineMartinRouchon:1992,FliessLevineMartinRouchon:1995}
and \cite{FliessLevineMartinRouchon:1999}). Flat continuous-time
systems have the characteristic feature that all system variables
can be parametrized by a flat output and its time derivatives. They
form an extension of the class of static feedback linearizable systems
and can be linearized by endogenous dynamic feedback. Their popularity
stems from the fact that a lot of physical systems possess the property
of flatness and that the knowledge of a flat output allows an elegant
solution to motion planning problems and a systematic design of tracking
controllers.

For nonlinear discrete-time systems, flatness can be defined analogously
to the continuous-time case. The main difference is that time derivatives
have to be replaced by forward-shifts. Like in the continuous-time
case, flat systems form an extension of static feedback linearizable
systems. The problem of static feedback linearization for discrete-time
systems is already solved, see \cite{Grizzle:1986}, \cite{Jakubczyk:1987}
and \cite{Aranda-BricaireKottaMoog:1996}. An important difference
to the continuous-time case is the existence of discrete-time systems
that can be linearized by exogenous dynamic feedback only. This fact
has been pointed out first in \cite{Aranda-BricaireMoog:2008}. The
corresponding linearizing output contains not only forward-shifts
but also backward-shifts of system variables. This gives rise to the
question whether the definition of discrete-time flatness should be
extended to both forward- and backward-shifts, as proposed in \cite{GuillotMillerioux2020}.
However, within this contribution, we follow \cite{KaldmaeKotta:2013},
\cite{KolarKaldmaeSchoberlKottaSchlacher:2016} and define flatness
in such a manner that it corresponds to the endogenous dynamic feedback
linearization problem. Therefore, we only consider forward-shifts
in the flat output. In order to avoid any confusion, we also use the
term endogenous flatness.

In general, the analysis of flat systems can be divided into two separate
tasks. First, we are interested in checking whether a system is flat
or not in order to clarify if flatness based control strategies can
be applied in principle. In \cite{Grizzle:1986} and \cite{NijmeijervanderSchaft:1990},
an efficient test for static feedback linearizable systems, which
is based on the computation of certain distributions, can be found.
As we have shown in \cite{KolarDiwoldSchoberl:2019}, this test can
be generalized to systems that possess the property of endogenous
flatness. The test is based on the results of \cite{KolarSchoberlDiwold:2019}.
It should be noted that for continuous-time flat systems no comparable
test is available so far. Second, in order to use the flatness property
for control strategies, the knowledge of a flat output as well as
the corresponding parameterization of all system variables is necessary.
For this purpose, the use of structurally flat normal forms (see e.g.
\cite{SchoberlSchlacher:2014}, \cite{Schoberl:2014} and \cite{KolarSchoberlSchlacher:2016-2})
has turned out to be helpful. Structurally flat normal forms allow
to read off the flat output, as well as a systematic construction
of the parameterization of all system variables. The most famous example
for such a normal form is the Brunovsky normal form. However, a transformation
to Brunovsky normal form is possible if and only if the system is
static feedback linearizable. In \cite{KolarDiwoldSchoberl:2019},
we have shown that every system that possesses the property of endogenous
flatness can be transformed into a structurally flat implicit normal
form. The main feature of this normal form is that the equations depend
on the system variables in a triangular manner. The reason for the
implicit character of this normal form is that the required coordinate
transformations are possibly more general than the usual state- and
input transformations. In the present contribution, we show that systems
with two inputs are an exception and can be transformed into a structurally
flat normal form by using state- and input transformations only. Thus,
the state representation is preserved. For this reason, we also use
the term explicit triangular normal form. We want to emphasize that
for flat continuous-time systems no comparable normal form exist.

The paper is organized as follows: In Section \ref{sec:Flatness-of-Time}
we recapitulate the concept of endogenous flatness and the corresponding
test according to \cite{KolarDiwoldSchoberl:2019}. In Section \ref{sec:useful_stuff}
we discuss certain coordinate transformations, which will be useful
later on. In Section \ref{sec:Explicit-Triangular-Form} we introduce
a structurally flat explicit triangular form. Then, we prove that
every two-input discrete-time system that possesses the property of
endogenous flatness can be transformed into such a representation
by successive state- and input transformations. Finally, in Section
\ref{sec:Example}, we illustrate our results by an example.

\section{Flatness of Discrete-Time Systems\label{sec:Flatness-of-Time}}

Throughout this contribution we consider discrete-time nonlinear systems
in explicit state representation of the form

\begin{equation}
x^{i,+}=f^{i}(x,u)\,,\quad i=1,\dots,n\label{eq:sysEq-1}
\end{equation}
with $\dim(x)=n$, $\dim(u)=m$ and smooth functions $f^{i}(x,u)$.
Geometrically, the system \eqref{eq:sysEq-1} can be interpreted as
a map $f$ from a manifold $\mathcal{X}\times\mathcal{U}$ with coordinates
$(x,u)$ to a manifold $\mathcal{X}^{+}$ with coordinates $x^{+}$.
Furthermore, we assume that the system meets $\Rank(\partial_{(x,u)}f)=n$,
which is a necessary condition for accessibility and consequently
also for flatness. Apart from this, we assume that the system possesses
no redundant inputs, i.e. $\Rank(\partial_{u}f)=m$, and define endogenous
flatness according to \cite{KolarDiwoldSchoberl:2019}.
\begin{defn}
\label{def:endflat}A system \eqref{eq:sysEq-1} possesses the property
of endogenous\footnote{In the rest of the paper we neglect the supplement ``endogenous''
and use only the term ``flat'' for such a system.} flatness around an equilibrium $(x_{0},u_{0})$, if there exists
an $m$-tuple of functions
\begin{equation}
y^{j}=\varphi^{j}(x,u,u_{[1]},u_{[2]}\dots,u_{[q]})\,,\hphantom{aa}j=1,\dots,m\,,\label{eq:flatOut}
\end{equation}
where $u_{[\alpha]}$ denotes the $\alpha$-th forward-shift of $u$,
such that the $n+m$ coordinate functions $x$ and $u$ can be expressed
locally by $y$ and forward-shifts of $y$ up to some finite order,
i.e.\footnote{We use the same notation as in \cite{KolarSchoberlDiwold:2019} and
\cite{KolarDiwoldSchoberl:2019} with a multi-index $R=(r_{1},\dots,r_{m})$.}
\[
\begin{aligned}x^{i} & =F_{x}^{i}(y,y_{[1]},y_{[2]},\dots,y_{[R-1]})\,,\hphantom{aa}i=1,\dots,n\\
u^{j} & =F_{u}^{j}(y,y_{[1]},y_{[2]},\dots,y_{[R]})\,,\hphantom{aaai}j=1,\dots,m\,.
\end{aligned}
\]
The $m$-tuple \eqref{eq:flatOut} is called a flat output.
\end{defn}

The test for flatness, as stated in \cite{KolarDiwoldSchoberl:2019},
is based on the construction of sequences of nested distributions
on $\mathcal{X}\times\mathcal{U}$ and $\mathcal{X}^{+}$. The construction
makes use of the system equations \eqref{eq:sysEq-1} and the map
\[
\pi:\mathcal{X}\times\mathcal{U}\rightarrow\mathcal{X}^{+}
\]
defined by
\[
x^{i,+}=x^{i}\,,\hphantom{aa}i=1,\dots,n\,.
\]
\begin{alg}\label{def:D_Delta_Def} \ \\
\textbf{Step $k=0$}: 
Define the distribution 
\begin{align}
\Delta_{0}=0
\notag
\end{align}
on $\mathcal{X}^{+}$ and
\begin{align*}
E_{0}=\pi_{*}^{-1}(\Delta_{0})=\Span\{\partial_{u}\}
\end{align*}
on $\mathcal{X}\times\mathcal{U}$. Then compute the largest subdistribution
\begin{align}
D_{0}\subset E_{0}
\label{eq:D0_def}
\end{align}
which is projectable\footnote{By projectable we mean that the pushforward $f_{*}(D_{0})$ is a well-defined distribution on $\mathcal{X}^+$ (see \cite{KolarDiwoldSchoberl:2019}).} with respect to the map $f$ of \eqref{eq:sysEq-1}. The distribution $D_0$ is involutive\footnote{The largest projectable subdistribution $D$ (with resp. to $f$) of a distribution $E$ is uniquely determined. Furthermore, if $E$ is involutive, then $D$ is also involutive (see \cite{KolarDiwoldSchoberl:2019}).} and its pushforward 
\begin{align}
\Delta_{1}=f_{*}(D_{0})
\notag
\end{align}
is a well-defined involutive distribution on $\mathcal{X}^+$.\\
\textbf{Step $k\geq1$}: 
Compute
\begin{align}
E_{k}=\pi_{*}^{-1}(\Delta_{k})
\label{eq:Def_Ek}
\end{align}
and the largest subdistribution 
\begin{align*}
D_{k}\subset E_{k}
\end{align*}
which is projectable with respect to the map $f$ of \eqref{eq:sysEq-1}. The distribution $D_k$ is involutive and its pushforward
\begin{align*}
\Delta_{k+1}=f_{*}(D_{k})
\end{align*}
is a well-defined involutive distribution on $\mathcal{X}^+$.\\
\textbf{Stop} if for some $k=\bar{k}$, 
\begin{align}
\dim(\text{\ensuremath{\Delta_{\bar{k}+1})=\dim(\Delta_{\bar{k}})}}\,.
\notag
\end{align}
\end{alg}

The procedure according to Algorithm \ref{def:D_Delta_Def} yields
a unique nested sequence of projectable and involutive distributions
\begin{equation}
D_{0}\subset D_{1}\subset\ldots\subset D_{\bar{k}-1}\label{eq:D_seq}
\end{equation}
 on $\mathcal{X}\times\mathcal{U}$ and a unique nested sequence of
involutive distributions
\begin{equation}
\Delta_{1}\subset\Delta_{2}\subset\ldots\subset\Delta_{\bar{k}}\label{eq:Delta_seq}
\end{equation}
 on $\mathcal{X}^{+}$ so that
\begin{equation}
f_{*}(D_{k})=\Delta_{k+1}\,,\quad k=0,\dots,\bar{k}-1\,.\label{eq:push_forward_Dk}
\end{equation}
Whether a system is flat or not can now be checked by the use of the
following theorem.
\begin{thm}
\label{thm:A-system-endog}A system \eqref{eq:sysEq-1} with $\Rank(\partial_{u}f)=m$
is flat if and only if $\dim(\Delta_{\bar{k}})=n$.
\end{thm}

For the proof we refer to \cite{KolarDiwoldSchoberl:2019}. The test
for flatness contains the test for static feedback linearizability
(see \cite{NijmeijervanderSchaft:1990}) as a special case. The only
difference is that the distributions \eqref{eq:D_seq}, according
to Algorithm \ref{def:D_Delta_Def}, are defined as the largest projectable
subdistributions $D_{k}\subset E_{k}$, while in the static feedback
linearizable case these distributions coincide, i.e. $D_{k}=E_{k}$.
\begin{thm}
\label{thm:staticFeedb}A system \eqref{eq:sysEq-1} with $\Rank(\partial_{u}f)=m$
is static feedback linearizable if and only if $D_{k}=E_{k}$, $k\geq0$
and $\dim(\Delta_{\bar{k}})=n$.
\end{thm}

Since all distributions $E_{k}$ must be completely projectable, the
test for static feedback linearizability is more restrictive. For
a system that meets Theorem \ref{thm:A-system-endog}, a single step
where $D_{k}\neq E_{k}$ can be interpreted as a defect in the test
of static feedback linearizability, as we will demonstrate by the
following example.

\begin{example}
\label{exa:FlatCheck}For the system

\begin{equation}
\begin{aligned}x^{5,+} & =x^{4}+x^{1}+x^{5}\\
x^{4,+} & =x^{1}(x^{4}+1)+x^{3}\\
x^{3,+} & =x^{1}+x^{2}\\
x^{2,+} & =u^{1}\\
x^{1,+} & =u^{2}
\end{aligned}
\label{eq:example_1}
\end{equation}
we obtain the sequence of distributions 
\begin{align*}
D_{0} & =\Span\{\partial_{u^{1}},\partial_{u^{2}}\}=E_{0}\\
D_{1} & =\Span\{\partial_{u^{1}},\partial_{u^{2}},\partial_{x^{2}}\}\subset E_{1}=\Span\{\partial_{u^{1}},\partial_{u^{2}},\partial_{x^{1}},\partial_{x^{2}}\}\\
D_{2} & =\Span\{\partial_{u^{1}},\partial_{u^{2}},\partial_{x^{2}},\partial_{x^{1}},\partial_{x^{3}}\}=E_{2}
\end{align*}
on $\mathcal{X}\times\mathcal{U}$ and 
\begin{align*}
\Delta_{1} & =\Span\{\partial_{x^{1,+}},\partial_{x^{2,+}}\}\\
\Delta_{2} & =\Span\{\partial_{x^{1,+}},\partial_{x^{2,+}},\partial_{x^{3,+}}\}\\
\Delta_{3} & =\Span\{\partial_{x^{1,+}},\partial_{x^{2,+}},\partial_{x^{3,+}},\partial_{x^{4,+}},\partial_{x^{5,+}}\}
\end{align*}
on $\mathcal{X}^{+}$. Despite the fact that $E_{1}$ is not completely
projectable, i.e. $D_{1}\neq E_{1}$, the distribution $\Delta_{3}$
meets $\dim(\Delta_{3})=n$ and the system possesses the weaker property
of flatness instead of static feedback linearizability. A flat output
is given by $y=(x^{4},x^{5})$.

Steps with $D_{k}\neq E_{k}$ may occur several times through the
algorithm. However, like in Example \ref{exa:FlatCheck}, for the
last distribution the relation $D_{\bar{k}-1}=E_{\bar{k}-1}$ holds.
Since we will use this relation in Section \ref{sec:Explicit-Triangular-Form},
we establish the following lemma.
\end{example}

\begin{lem}
\label{lem:EProj}For a flat system \eqref{eq:sysEq-1} the distribution
$E_{\bar{k}-1}$ is completely projectable, i.e. $D_{\bar{k}-1}=E_{\bar{k}-1}$.
\end{lem}

\begin{proof}
The pushforward of the last distribution $D_{\bar{k}-1}$ meets
\[
f_{*}(D_{\bar{k}-1})=\Delta_{\bar{k}}
\]
with
\[
\Delta_{\bar{k}}=\Span\{\partial_{x^{1,+}},\dots,\partial_{x^{n,+}}\}\,.
\]
Because of $D_{\bar{k}-1}\subset E_{\bar{k}-1}$ and $\dim(\Delta_{\bar{k}})=\dim(\mathcal{X}^{+})=n$,
we also get
\[
f_{*}(E_{\bar{k}-1})=\Span\{\partial_{x^{1,+}},\dots,\partial_{x^{n,+}}\}\,.
\]
Thus, $E_{\bar{k}-1}$ is projectable and according to the definition
of $D_{\bar{k}-1}$ we have $D_{\bar{k}-1}=E_{\bar{k}-1}$.
\end{proof}

\section{Coordinate Transformations\label{sec:useful_stuff}}

For static feedback linearizable systems, the sequences of distributions
\eqref{eq:D_seq} and \eqref{eq:Delta_seq} can be straightened out
simultaneously by suitable state transformations. As shown in \cite{NijmeijervanderSchaft:1990},
in such coordinates the system \eqref{eq:sysEq-1} exhibits a triangular
form. The primary objective of the present paper is to prove that
every flat system with two inputs allows a similar explicit\footnote{By the term ``explicit'' we refer to a state representation \eqref{eq:sysEq-1},
in order to distinguish it from the implicit triangular representation
discussed in \cite{KolarDiwoldSchoberl:2019}.} triangular representation. The key tool is again to straighten out
the sequences of distributions \eqref{eq:D_seq} and \eqref{eq:Delta_seq}
simultaneously (in Example \ref{exa:FlatCheck} we already have such
special coordinates). However, now we will need in general both state-
and input transformations. For systems with more than two inputs,
there is no guarantee that an explicit triangular representation exists
at all.

Before we state our main results in Section \ref{sec:Explicit-Triangular-Form},
we discuss certain state- and input transformations which will be
useful for straightening out the sequences of distributions. In order
to preserve an explicit system representation like \eqref{eq:sysEq-1},
within this contribution we restrict ourselves to state- and input
transformations
\begin{gather}
\begin{aligned}\hat{x}^{i} & =\Phi_{x}^{i}(x)\,,\hphantom{aaa}i=1,\dots,n\\
\hat{u}^{j} & =\Phi_{u}^{j}(x,u)\,,\hphantom{a}j=1,\dots,m\,,
\end{aligned}
\label{eq:expTransf}
\end{gather}
where both $x$ and $x^{+}$ are transformed equally. The transformed
system is given by
\begin{equation}
\hat{x}^{i,+}=\underset{f^{i}(\hat{x},\hat{u})}{\underbrace{\Phi_{x}^{i}(x^{+})\circ f(x,u)\circ\Phi^{-1}(\hat{x},\hat{u})}}\,,\hphantom{a}i=1,\dots,n
\end{equation}
where $\Phi^{-1}(\hat{x},\hat{u})$ denotes the inverse of \eqref{eq:expTransf}.
Like in the static feedback linearizable case, the first step in achieving
a triangular representation is to straighten out the sequence \eqref{eq:Delta_seq}.
Since \eqref{eq:Delta_seq} is a nested sequence of involutive distributions
on $\mathcal{X}^{+}$, by an extension of the Frobenius theorem there
exists a state transformation
\begin{equation}
(\bar{x}_{1},\dots,\bar{x}_{\bar{k}})=\Phi_{x}(x)\,,\label{eq:state_transf_delta}
\end{equation}
with $\dim(\bar{x}_{k})=\dim(\Delta_{k})-\dim(\Delta_{k-1})$, which
straightens out the distributions
\begin{gather}
\begin{aligned}\Delta_{1} & =\Span\{\partial_{\bar{x}_{1}^{+}}\}\\
\Delta_{2} & =\Span\{\partial_{\bar{x}_{1}^{+}},\partial_{\bar{x}_{2}^{+}}\}\\
 & \hphantom{a}\vdots\\
\Delta_{\bar{k}} & =\Span\{\partial_{\bar{x}_{1}^{+}},\partial_{\bar{x}_{2}^{+}},\dots,\partial_{\bar{x}_{\bar{k}}^{+}}\}\,
\end{aligned}
\label{eq:straight_Delta}
\end{gather}
simultaneously. The system in new coordinates reads as\footnote{Note that subsequently we will use the bar notation for system representations
where the $\Delta$-sequence is already straightened out.}
\begin{gather}
\begin{aligned}\bar{x}_{k}^{+} & =f_{k}(\bar{x},u)\,,\end{aligned}
\hphantom{a}k=1,\dots,\bar{k}\label{eq:sys_in_straight_Delta}
\end{gather}
and meets 
\begin{equation}
f_{*}(D_{k-1})=\Span\{\partial_{\bar{x}_{1}^{+}},\partial_{\bar{x}_{2}^{+}},\dots,\partial_{\bar{x}_{k}^{+}}\}
\end{equation}
for $k=1,\dots,\bar{k}$. Additionally, from the definition of $E_{k}$
according to \eqref{eq:Def_Ek} and \eqref{eq:straight_Delta}, it
follows automatically that $E_{k}$ is also straightened out and reads
as
\begin{equation}
E_{k}=\pi_{*}^{-1}(\Delta_{k})=\Span\{\partial_{\bar{x}_{k}},\dots,\partial_{\bar{x}_{1}},\partial_{u}\}\label{eq:straight_E}
\end{equation}
for $k=0,\dots,\bar{k}-1$. The $D$-sequence, in contrast, is only
straightened out automatically if the system possesses the stronger
property of static feedback linearizability, since then $D_{k}=E_{k}$.
Thus, for finding coordinates that straighten out both sequences
of distributions \eqref{eq:D_seq} and \eqref{eq:Delta_seq}, the
transformation \eqref{eq:state_transf_delta} alone is not sufficient.
Therefore, we need an additional transformation that straightens out
the $D$-sequence while the $\Delta$-sequence remains straightened
out. In the following, we introduce transformations that meet the
latter condition.
\begin{lem}
\label{lem:presFlat}State- and input transformations of the form

\begin{gather}
\begin{aligned}\hat{x}_{k} & =\Phi_{\bar{x},k}(\bar{x}_{k},\dots,\bar{x}_{\bar{k}})\,,\hphantom{a}k=1,\dots,\bar{k}\\
\hat{u} & =\Phi_{u}(\bar{x},u)
\end{aligned}
\label{eq:delta_preserv_state_transf}
\end{gather}
preserve the structure \eqref{eq:straight_Delta} of the sequence
of distributions \eqref{eq:Delta_seq}, i.e. 
\begin{gather*}
\begin{aligned}\Delta_{1} & =\Span\{\partial_{\hat{x}_{1}^{+}}\}\\
\Delta_{2} & =\Span\{\partial_{\hat{x}_{1}^{+}},\partial_{\hat{x}_{2}^{+}}\}\\
 & \hphantom{a}\vdots\\
\Delta_{\bar{k}} & =\Span\{\partial_{\hat{x}_{1}^{+}},\partial_{\hat{x}_{2}^{+}},\dots,\partial_{\hat{x}_{\bar{k}}^{+}}\}\,.
\end{aligned}
\end{gather*}
\end{lem}

The proof follows from the triangular structure of the state transformation
of \eqref{eq:delta_preserv_state_transf}. The input transformation
does not affect the $\Delta$-sequence.

For systems with two inputs the distributions of the $D$-sequence
have a very special structure. Since we will deal with two-input systems
in Section \ref{subsec:Explicit-Triangular-Form}, we state the following
important lemma.
\begin{lem}
\label{lem:An-involutive-distribution}Consider an $n$-dimensional
manifold $\mathcal{Z}$ with coordinates $\zcord=(\zcord^{1},\dots,\zcord^{n})$
and an involutive distribution
\begin{equation}
D=\Span\{\partial_{\zcord^{1}},\dots,\partial_{\zcord^{k-1}},\partial_{\zcord^{i}}+\alpha(\zcord)\partial_{\zcord^{j}}\}\label{eq:D_form}
\end{equation}
for some $i,j\geq k$. There exists a transformation
\begin{equation}
\hat{\zcord}^{j}=\Phi^{j}(\zcord^{k},\dots,\zcord^{n})\label{eq:z_transf}
\end{equation}
of the coordinate $\zcord^{j}$ such that
\begin{equation}
D=\Span\{\partial_{\zcord^{1}},\dots,\partial_{\zcord^{k-1}},\partial_{\zcord^{i}}\}\,.\label{eq:straight_Z}
\end{equation}
\end{lem}

The proof can be found in the appendix. For two-input systems we will
encounter distributions $D_{k}$ of the form \eqref{eq:D_form} on
the manifold $\mathcal{X}\times\mathcal{U}$, and straighten them
out by state- or input transformations of the type \eqref{eq:z_transf}.
These transformations will also exhibit the structure-preserving form
\eqref{eq:delta_preserv_state_transf} with respect to the $\Delta$-sequence.

\section{Explicit Triangular Form\label{sec:Explicit-Triangular-Form}}

In \cite{NijmeijervanderSchaft:1990} it is shown how a static feedback
linearizable system can be transformed into Brunovsky normal form.
In the first step, a state transformation is performed that straightens
out the sequences of distributions simultaneously. This yields an
explicit triangular system representation which can be interpreted
as a composition of smaller subsystems. With respect to the inputs
of these subsystems, there may occur redundancies. Following \cite{NijmeijervanderSchaft:1990},
further state- and input transformations are successively performed
in order to obtain the Brunovsky normal form. The above mentioned
redundancies appear if the chains of the Brunovsky normal form have
different lengths.

For flat systems that are not static feedback linearizable, a transformation
to Brunovsky normal form is not possible. Thus, we introduce a more
general structurally flat explicit triangular form that can be obtained
by straightening out the $D$- and $\Delta$-sequences by suitable
state- and input transformations. For two-input systems, we prove
that such a transformation which straightens out both sequences of
distributions simultaneously always exists. Subsequently, similar
to the static feedback linearizable case, redundant inputs of the
subsystems can be eliminated by further structure-preserving state-
and input transformations. For the resulting system representation
we use the term explicit triangular normal form. It allows to read
off a flat output and the corresponding parameterizing map, according
to Definition \ref{def:endflat}, in a systematic way.

\subsection{Explicit Triangular Form for Multi-Input Systems}

In the following we present an explicit triangular representation
for flat systems. Note, we do not refer to it as a normal form, since
for systems with an arbitrary number of inputs the existence of such
coordinates is not guaranteed in general.
\begin{thm}
\label{thm:Triangular_Form}Assume there exists a state- and input
transformation
\begin{gather}
\begin{aligned}\hat{x} & =\Phi_{x}(x)\\
\hat{u} & =\Phi_{u}(x,u)
\end{aligned}
\label{eq:transf_to_triangular}
\end{gather}
that straightens out the sequences \eqref{eq:D_seq} and \eqref{eq:Delta_seq}
simultaneously, i.e.
\begin{equation}
\Delta_{j}=\Span\{\partial_{\hat{x}_{1}^{+}},\dots,\partial_{\hat{x}_{j}^{+}}\}\,,\hphantom{a}j=1,\dots,\bar{k}\label{eq:straight_Delta-1}
\end{equation}
with $\dim(\hat{x}_{j})=\dim(\Delta_{j})-\dim(\Delta_{j-1})$ and
\begin{equation}
D_{k}=\Span\{\partial_{z_{0}},\dots,\partial_{z_{k}}\}\,,\hphantom{a}k=0,\dots,\bar{k}-1\,.\label{eq:D_straightz}
\end{equation}
Here $z_{k}$ denotes a selection of components of $(\hat{u},\hat{x}_{1},\dots,\hat{x}_{k})$
with $\dim(z_{0})=\dim(D_{0})$ and $\dim(z_{k})=\dim(D_{k})-\dim(D_{k-1})$
for $k=1,\text{\ensuremath{\dots,\bar{k}-1}}$. In such coordinates,
the system \eqref{eq:sysEq-1} has the triangular form
\begin{gather}
\begin{aligned}\hat{x}_{\bar{k}}^{+} & =f_{\bar{k}}(\hat{x}_{\bar{k}},z_{\bar{k}-1})\\
\hat{x}_{\bar{k}-1}^{+} & =f_{\bar{k}-1}(\hat{x}_{\bar{k}},z_{\bar{k}-1},z_{\bar{k}-2})\\
 & \hphantom{a}\vdots\\
\hat{x}_{1}^{+} & =f_{1}(\hat{x}_{\text{\ensuremath{\bar{k}}}},z_{\bar{k}-1},\dots,z_{0})
\end{aligned}
\label{eq:Expicit_Triangular_I_form}
\end{gather}
with
\begin{equation}
\hat{x}_{k}\subset(z_{k},\dots,z_{\bar{k}-1})\,,\quad k=1,\dots,\bar{k}-1\label{eq:x_sub_z}
\end{equation}
and meets
\begin{equation}
\Rank(\partial_{z_{j-1}}f_{j})=\dim(\hat{x}_{j})\,,\quad j=1,\dots,\bar{k}\,.\label{eq:Rankcond}
\end{equation}
\end{thm}

The proof can be found in the appendix. Hereinafter, we state the
intention of defining $z_{k}$. Since the $\Delta$-sequence is straightened
out, likewise is $E_{k}=\Span\{\partial_{\hat{x}_{k}},\dots,\partial_{\hat{x}_{1}},\partial_{\hat{u}}\}$.
Furthermore, by assumption the $D$-sequence is also straightened
out, and due to $D_{k}\subset E_{k}$ it follows that $D_{k}$ might
not contain all components of $\partial_{\hat{x}_{k}},\dots,\partial_{\hat{x}_{1}},\partial_{\hat{u}}$.
Therefore, we introduce the variable $z_{k}$, which acts as placeholder
and describes states and/or inputs of $(\hat{u},\hat{x}_{1},\dots,\hat{x}_{k})$
so that $D_{k}$ reads as \eqref{eq:D_straightz}. It is important
to mention that the variables $z_{0},\dots,z_{\bar{k}-1}$ contain
all inputs and states except $\hat{x}_{\bar{k}}$ (see the proof in
the appendix). Note, the system \eqref{eq:Expicit_Triangular_I_form}
is still in an explicit state representation \eqref{eq:sysEq-1},
as we can always replace $z_{k}$ by the corresponding states and
inputs. We clarify the definition of $z_{k}$ using the system of
Example \ref{exa:FlatCheck}.
\begin{example}
\label{exa:Consider-the-system}Consider the system \eqref{eq:example_1}
of Example \ref{exa:FlatCheck}. In this example, both the $D$- and
the $\Delta$-sequence are already straightened out. Thus, the coordinate
transformation of Theorem \ref{thm:Triangular_Form} is just a renaming
\[
\begin{aligned}\hat{x}_{3}^{1} & =x^{5} & \hat{x}_{2}^{1} & =x^{3} & \hat{x}_{1}^{1} & =x^{2} & \hat{u}^{1} & =u^{1}\\
\hat{x}_{3}^{2} & =x^{4} &  &  & \hat{x}_{1}^{2} & =x^{1} & \hat{u}^{2} & =u^{2}\,.
\end{aligned}
\]
According to Theorem \ref{thm:Triangular_Form} we define 
\begin{equation}
\begin{aligned}z_{0} & =(\hat{u}^{1},\hat{u}^{2})\,, & z_{1} & =(\hat{x}_{1}^{1})\,, & z_{2} & =(\hat{x}_{2}^{1},\hat{x}_{1}^{2})\end{aligned}
\label{eq:z_example1}
\end{equation}
such that the $D$-sequence of distributions reads as
\begin{align*}
D_{0} & =\Span\{\partial_{z_{0}}\}\\
D_{1} & =\Span\{\partial_{z_{0}},\partial_{z_{1}}\}\\
D_{2} & =\Span\{\partial_{z_{0}},\partial_{z_{1}},\partial_{z_{2}}\}
\end{align*}
and the system follows as
\[
\begin{aligned}x_{3}^{1,+} & =x_{3}^{2}+z_{2}^{2}+x_{3}^{1}\\
x_{3}^{2,+} & =z_{2}^{2}(x_{3}^{2}+1)+z_{2}^{1}\\
x_{2}^{1,+} & =z_{2}^{2}+z_{1}^{1}\\
x_{1}^{1,+} & =z_{0}^{1}\\
x_{1}^{2,+} & =z_{0}^{2}\,.
\end{aligned}
\]
The system has the structure of \eqref{eq:Expicit_Triangular_I_form},
and a flat output is given by $y=(x_{3}^{1},x_{3}^{2}).$
\end{example}

It is important to emphasize that for systems with $m>2$ inputs the
existence of a state- and input transformation \eqref{eq:transf_to_triangular}
that straightens out both sequences \eqref{eq:D_seq} and \eqref{eq:Delta_seq}
simultaneously is not guaranteed. Thus, an explicit triangular form
\eqref{eq:Expicit_Triangular_I_form} does not necessarily exist.
However, at least a transformation into an implicit triangular form
as discussed in \cite{KolarDiwoldSchoberl:2019} is always possible.

\subsection{Explicit Triangular Form for Two-Input Systems\label{subsec:Explicit-Triangular-Form}}

In the following we restrict ourselves to flat systems with two inputs
and state our main result.
\begin{thm}
\label{thm:two_input_always}A two-input flat system \eqref{eq:sysEq-1}
is locally transformable into an explicit triangular representation
\eqref{eq:Expicit_Triangular_I_form}.
\end{thm}

According to Theorem \ref{thm:Triangular_Form}, we must show that
there exists a state- and input transformation \eqref{eq:transf_to_triangular}
which straightens out the sequences of distributions \eqref{eq:D_seq}
and \eqref{eq:Delta_seq} simultaneously. We start with the system
representation \eqref{eq:sys_in_straight_Delta}, where the $\Delta$-sequence
has already been straightened out by a suitable state transformation.
Next, we want to straighten out the $D$-sequence step by step, starting
with $D_{0}$. For this purpose, we can exploit the fact that for
systems with two inputs the dimension of these distributions grows
in every step by either one or two. In the first case, the distribution
$D_{k}$ is of the form \eqref{eq:D_form} and can be straightened
out by a transformation \eqref{eq:z_transf}, whereas in the latter
case the distribution $D_{k}$ is already straightened out and no
transformation is required.

The following algorithm straightens out the $D$-sequence step by
step with state- and input transformations that preserve the structure
of the $\Delta$-sequence according to Lemma \ref{lem:presFlat}.
In every step $k$, after performing the transformation the corresponding
states or inputs are renamed by $z_{k}$ and $z_{k,c}$. As mentioned
before, $z_{k}$ is just a selection of states and inputs so that
\[
D_{k}=\Span\{\partial_{z_{0}},\dots,\partial_{z_{k}}\}\,,
\]
whereas $z_{k,c}$ denotes the complementary state or input so that
\[
E_{k}=\Span\{\partial_{z_{0}},\dots,\partial_{z_{k}},\partial_{z_{k,c}}\}\,.
\]
To keep the successive transformations readable, after each step $k$
we return from the hat notation for the transformed variables again
to the bar notation. However, for the final system representation
after the last step we use the hat notation.\begin{alg}\ \\
\label{alg:theAlg}
\textbf{Step $k=0$}, 
we distinguish between the two cases:
\begin{itemize}
\item[a)] If the entire input distribution is projectable, i.e. $D_{0}=E_{0}$, then there is no need for an input transformation because $E_{0}$ is already straightened out. We define $z_{0}=(u^{1},u^{2})$ and $z_{0,c}$ is empty.
\item[b)] If $D_{0}\neq E_{0}$, then 
\begin{equation}
D_{0}=\Span\{\alpha(\bar{x},u)\partial_{u^{1}}+\partial_{u^{2}}\}\,,
\notag
\end{equation}
up to a renumbering of the components of $u$. According to Lemma \ref{lem:An-involutive-distribution}, there exists an input transformation $\hat{u}^{1}=\Phi_{u^{1}}(\bar{x},u)$ such that $D_{0}=\Span\{\partial_{u^{2}}\}$. We define $z_{0}=u^{2}$ and $z_{0,c}=\hat{u}^{1}$.
\end{itemize}
Finally, the distributions are given by
\begin{align}
D_{0}&=\Span\{\partial_{z_{0}}\}\notag\\
E_{0}&=\Span\{\partial_{z_{0}},\partial_{z_{0,c}}\}\notag\,.
\end{align}

\textbf{Step $k=1,\dots,\bar{k}-1$}, we repeat the procedure with the distribution 
\begin{align*}
D_{k}\subset E_{k}=\Span\{\partial_{z_{0}},\dots,\partial_{z_{k-1}},\partial_{z_{k-1,c}},\partial_{\bar{x}_{k}}\}\,.
\end{align*}
It can be shown that the dimensions of $z_{k-1,c}$ and $\bar{x}_{k}$ meet 
\begin{align*}
\dim(z_{k-1,c})&\leq 1\\
\dim(\bar{x}_{k})&\geq 1\\
\dim(z_{k-1,c})+\dim(\bar{x}_{k})&\leq 2\,.
\end{align*}
Thus, in every step we must distinguish three cases:
\begin{itemize}

\item[a)] If the entire distribution $E_k$ is projectable, i.e. $D_k=E_k$, then there is no need for a transformation because $E_k$ is already straightened out. We define $z_{k}=(\bar{x}_{k},z_{k-1,c})$ and $z_{k,c}$ is empty.

\item[b)] If $D_{k}\neq E_{k}$ and $\dim(\bar{x}_{k})=2$, then $z_{k-1,c}$ is empty and
\begin{align*}
D_{k}=\Span\{\partial_{z_{0}},\dots,\partial_{z_{k-1}},\alpha(z,\bar{x})\partial_{\bar{x}_{k}^{1}}+\partial_{\bar{x}_{k}^{2}}\}\,,
\end{align*}
up to a renumbering of the components of $\bar{x}_{k}$. According to Lemma \ref{lem:An-involutive-distribution}, there exists a state transformation
\begin{align}
\hat{x}_{k}^{1}=\Phi_{\bar{x}_{k}^{1}}(\bar{x}_{k},\dots,\bar{x}_{\bar{k}})
\label{eq:state_transf_spec_1}
\end{align}
such that $D_{k}=\Span\{\text{\ensuremath{\partial_{z_{0}},\dots,\partial_{z_{k-1}},\partial_{\bar{x}_{k}^{2}}}}\}$. We define $z_{k}=\bar{x}_{k}^{2}$ and $z_{k,c}=\hat{x}_{k}^{1}$.

\item[c)] If $D_{k}\neq E_{k}$ and $\dim(\bar{x}_{k})=1$, then necessarily also $\dim(z_{k-1,c})=1$. Otherwise, we would have $D_k=E_k$ and case (a) would apply. Thus, 
\begin{align*}
D_{k}=\Span\{&\partial_{z_{0}},\dots,\partial_{z_{k-1}}, \\&\alpha(z,z_{k-1,c},\bar{x})\partial_{z_{k-1,c}}+\partial_{\bar{x}_{k}}\}\,.
\end{align*}
According to Lemma \ref{lem:An-involutive-distribution}, there exists a transformation
\begin{align}
\label{eq:state_inp_transf_spec_2}
\hat{z}_{k-1,c}=\Phi_{z_{k-1,c}}(\bar{x}_{k},\dots,\bar{x}_{\bar{k}},z_{k-1,c})
\end{align}
such that $D_{k}=\text{span}\{\partial_{z_{0}},\dots,\partial_{z_{k-1}},\partial_{\bar{x}_{k}}\}$. Since $z_{\bar{k}-1,c}$ could represent both an input or state variable of the system, the transformation is either an input- or a state transformation. We define $z_{k}=\bar{x}_{k}$ and $z_{k,c}=\hat{z}_{k-1,c}$.
\end{itemize}
Finally, the distributions are given by
\begin{align}
D_{k}&=\Span\{\partial_{z_{0}},\dots,\partial_{z_{k}}\}\notag\\
E_{k}&=\Span\{\partial_{z_{0}},\dots,\partial_{z_{k}},\partial_{z_{k,c}}\}\,.\notag
\end{align}
\end{alg}

Within the algorithm, only state transformations \eqref{eq:state_transf_spec_1}
and input- or state transformations \eqref{eq:state_inp_transf_spec_2}
are performed. Since both are of the structure-preserving form \eqref{eq:delta_preserv_state_transf},
the resulting entire transformation law is also of the form \eqref{eq:delta_preserv_state_transf}
and the $\Delta$-sequence remains straightened out. Thus, after the
last step of the algorithm, both sequences of distributions are straightened
out according to \eqref{eq:straight_Delta-1} and \eqref{eq:D_straightz}
in Theorem \ref{thm:Triangular_Form}.
\begin{rem}
In the last step $k=\bar{k}-1$, according to Lemma \ref{lem:EProj}
the distribution $E_{\bar{k}-1}$ is completely projectable and case
(a) applies. Consequently, the last distribution meets $D_{\bar{k}-1}=E_{\bar{k}-1}$
and $z_{\bar{k}-1,c}$ is empty. Thus, it is ensured that the variables
$z_{0},\dots,z_{\bar{k}-1}$ indeed contain all inputs and states
except $\hat{x}_{\bar{k}}$.

\end{rem}

\subsection{Explicit Triangular Normal Form for Two-Input Systems}

The explicit triangular form \eqref{eq:Expicit_Triangular_I_form}
consists of the $\bar{k}$ subsystems
\begin{gather}
\begin{aligned}\hat{x}_{\bar{k}}^{+} & =f_{\bar{k}}(\hat{x}_{\bar{k}},z_{\bar{k}-1})\\
\hat{x}_{\bar{k}-1}^{+} & =f_{\bar{k}-1}(\hat{x}_{\bar{k}},z_{\bar{k}-1},z_{\bar{k}-2})\\
 & \hphantom{a}\vdots\\
\hat{x}_{k}^{+} & =f_{k}(\hat{x}_{\bar{k}},z_{\bar{k}-1},\dots,z_{k-1})
\end{aligned}
\label{eq:Explicit_subsys}
\end{gather}
with $k=1,\dots,\bar{k}$. The parameterization of the system variables
of the system \eqref{eq:Expicit_Triangular_I_form} by the flat output
can be obtained by determining step by step the parameterization of
the system variables of the subsystems \eqref{eq:Explicit_subsys},
starting with the topmost subsystem
\begin{equation}
\hat{x}_{\bar{k}}^{+}=f_{\bar{k}}(\hat{x}_{\bar{k}},z_{\bar{k}-1})\,.\label{eq:topmost-1}
\end{equation}
If $\dim(\hat{x}_{k})=\dim(z_{k-1})$ for all $k=1,\dots,\bar{k}$,
then due to the rank conditions \eqref{eq:Rankcond} this is particularly
simple, and $y=\hat{x}_{\bar{k}}$ with $\dim(\hat{x}_{\bar{k}})=m$
is a flat output (see Example \ref{exa:Consider-the-system}). By
applying the implicit function theorem to the topmost subsystem \eqref{eq:topmost-1}
we immediately get the parameterization of the variables $z_{\bar{k}-1}$.
Next, since $z_{\bar{k}-1}$ contains the state variables $\hat{x}_{\bar{k}-1}$
(see \eqref{eq:x_sub_z}), by applying the implicit function theorem
to the equations
\[
\hat{x}_{\bar{k}-1}^{+}=f_{\bar{k}-1}(\hat{x}_{\bar{k}},z_{\bar{k}-1},z_{\bar{k}-2})
\]
we get the parameterization of the variables $z_{\bar{k}-2}$. Continuing
this procedure finally yields the parameterization of all system variables
by the flat output $y=\hat{x}_{\bar{k}}$. However, if $\dim(\hat{x}_{\bar{k}})<m$,
then for at least one $k\in\{2,\dots,\bar{k}\}$ we have $\dim(\hat{x}_{k})<\dim(z_{k-1})$,
which means that the equations 
\[
\hat{x}_{k}^{+}=f_{k}(\hat{x}_{\bar{k}},z_{\bar{k}-1},\dots,z_{k-1})
\]
cannot be solved for all components of $z_{k-1}$. In this case, the
subsystem \eqref{eq:Explicit_subsys} has redundant inputs, and in
addition to $\hat{x}_{\bar{k}}$ the flat output has further components.
The redundant inputs can be eliminated from the subsystem by suitable
coordinate transformations of the structure-preserving form \eqref{eq:delta_preserv_state_transf}.
The flat output of the complete system \eqref{eq:Expicit_Triangular_I_form}
consists of $\hat{x}_{\bar{k}}$ and the eliminated redundant inputs
of all subsystems.

For systems with two inputs there can occur exactly two cases. If
$\dim(\hat{x}_{\bar{k}})=2$, then $y=\hat{x}_{\bar{k}}$ is a flat
output and none of the subsystems has redundant inputs. Otherwise,
if $\dim(\hat{x}_{\bar{k}})=1$, then there is exactly one $k\in\{2,\dots,\bar{k}\}$
with $\dim(\hat{x}_{k})=1<\dim(z_{k-1})=2$. Thus, the corresponding
subsystem \eqref{eq:Explicit_subsys} has one redundant input. By
the use of the regular transformation
\begin{align}
\bar{z}_{k-1}^{2} & =f_{k}(\hat{x}_{\bar{k}},z_{\bar{k}-1},\dots,z_{k-1})\,,\label{eq:red_input_transf}
\end{align}
which is either an input- or state transformation but still of type
\eqref{eq:delta_preserv_state_transf}, the subsystem reads as 
\[
\begin{aligned}\hat{x}_{\bar{k}}^{+} & =f_{\bar{k}}(\hat{x}_{\bar{k}},z_{\bar{k}-1})\\
\hat{x}_{\bar{k}-1}^{+} & =f_{\bar{k}-1}(\hat{x}_{\bar{k}},z_{\bar{k}-1},z_{\bar{k}-2})\\
 & \hphantom{a}\vdots\\
\hat{x}_{k}^{+} & =\bar{z}_{k-1}^{2}
\end{aligned}
\]
and is independent of $\text{\ensuremath{z_{k-1}^{1}}}$. The flat
output of the complete system is given by $y=(\hat{x}_{\bar{k}},z_{k-1}^{1})$.
After the elimination of the possibly occurring redundant input of
the subsystem \eqref{eq:Explicit_subsys} by the coordinate transformation
\eqref{eq:red_input_transf}, we refer to the resulting system representation
as explicit triangular normal form for two-input systems.

\section{Example\label{sec:Example}}

In this section, we demonstrate our results with an example already
discussed in \cite{KolarDiwoldSchoberl:2019}. The system reads as
\begin{equation}
\begin{aligned}x^{1,+} & =\tfrac{x^{2}+x^{3}+3x^{4}}{u^{1}+2u^{2}+1}\\
x^{2,+} & =x^{1}(x^{3}+1)(u^{1}+2u^{2}-3)+x^{4}-3u^{2}\\
x^{3,+} & =u^{1}+2u^{2}\\
x^{4,+} & =x^{1}(x^{3}+1)+u^{2}\,,
\end{aligned}
\label{eq:example_automa}
\end{equation}
and the sequences of distributions \eqref{eq:D_seq} and \eqref{eq:Delta_seq}
are given by
\begin{align*}
D_{0} & =\Span\{-2\partial_{u^{1}}+\partial_{u^{2}}\}\subset E_{0}=\Span\{\partial_{u^{1}},\partial_{u^{2}}\}\\
D_{1} & =\Span\{\partial_{u^{1}},\partial_{u^{2}},-3\partial_{x^{2}}+\partial_{x^{4}}\}=E_{1}\\
D_{2} & =\Span\{\partial_{u^{1}},\partial_{u^{2}},-3\partial_{x^{2}}+\partial_{x^{4}},\tfrac{x^{1}}{x^{3}+1}\partial_{x^{1}}-\partial_{x^{3}},\\
 & \hphantom{aaaaaaaaaaaaaaaaaa}\tfrac{2x^{1}}{x^{3}+1}\partial_{x^{1}}-2\partial_{x^{3}}-\partial_{x^{4}}\}=E_{2}
\end{align*}
on $\mathcal{X}\times\mathcal{U}$ and 
\begin{align*}
\Delta_{1} & =\Span\{-3\partial_{x^{2,+}}+\partial_{x^{4,+}}\}\\
\Delta_{2} & =\Span\{-3\partial_{x^{2,+}}+\partial_{x^{4,+}},\tfrac{x^{1,+}}{x^{3,+}+1}\partial_{x^{1,+}}-\partial_{x^{3,+}},\\
 & \hphantom{aaaaaaaaaaaaaa}\tfrac{2x^{1,+}}{x^{3,+}+1}\partial_{x^{1,+}}-2\partial_{x^{3,+}}-\partial_{x^{4,+}}\}\\
\Delta_{3} & =\Span\{\partial_{x^{1,+}},\partial_{x^{2,+}},\partial_{x^{3,+}},\partial_{x^{4,+}}\}
\end{align*}
on $\mathcal{X}^{+}$. Following the procedure of Section \ref{sec:Explicit-Triangular-Form},
first we straighten out the $\Delta$-sequence by a state transformation
of the form \eqref{eq:state_transf_delta} with $\bar{x}_{1}=\bar{x}_{1}^{1}$,
$\bar{x}_{2}=(\bar{x}_{2}^{1},\bar{x}_{2}^{2})$ and $\bar{x}_{3}=\bar{x}_{3}^{1}$.
With the transformation
\[
\begin{aligned}\bar{x}_{3}^{1} & =x^{1}(x^{3}+1) & \bar{x}_{2}^{1} & =x^{2}+3x^{4} & \bar{x}_{1}^{1} & =x^{4}\,,\\
 &  & \bar{x}_{2}^{2} & =x^{3}
\end{aligned}
\]
the $\Delta$-sequence reads as
\begin{align*}
\Delta_{1} & =\Span\{\partial_{\bar{x}_{1}^{1,+}}\}\\
\Delta_{2} & =\Span\{\partial_{\bar{x}_{1}^{1,+}},\partial_{\bar{x}_{2}^{1,+}},\partial_{\bar{x}_{2}^{2,+}}\}\\
\Delta_{3} & =\Span\{\partial_{\bar{x}_{1}^{1,+}},\partial_{\bar{x}_{2}^{1,+}},\partial_{\bar{x}_{2}^{2,+}},\partial_{\bar{x}_{3}^{1,+}}\}\,,
\end{align*}
and the system in new coordinates is given by
\[
\begin{aligned}\bar{x}_{3}^{1,+} & =\bar{x}_{2}^{1}+\bar{x}_{2}^{2}\\
\bar{x}_{2}^{1,+} & =\bar{x}_{1}^{1}+\bar{x}_{3}^{1}(u^{1}+2u^{2})\\
\bar{x}_{2}^{2,+} & =u^{1}+2u^{2}\\
\bar{x}_{1}^{1,+} & =\bar{x}_{3}^{1}+u^{2}\,.
\end{aligned}
\]
The $D$-sequence in new coordinates reads as
\begin{align*}
D_{0} & =\Span\{-2\partial_{u^{1}}+\partial_{u^{2}}\}\subset E_{0}\\
D_{1} & =\Span\{\partial_{u^{1}},\partial_{u^{2}},\partial_{\bar{x}_{1}^{1}}\}=E_{1}\\
D_{2} & =\Span\{\partial_{u^{1}},\partial_{u^{2}},\partial_{\bar{x}_{1}^{1}},\partial_{\bar{x}_{2}^{1}},\partial_{\bar{x}_{2}^{2}}\}=E_{2}\,.
\end{align*}
Next, we use Algorithm \ref{alg:theAlg} in order to straighten out
the $D$-sequence and transform the system into the explicit triangular
representation \eqref{eq:Expicit_Triangular_I_form}. Due to the fact
that $E_{0}$ is not completely projectable, the case (b) applies
and we need to perform an input transformation
\[
\hat{u}^{1}=u^{1}+2u^{2}
\]
which yields $D_{0}=\Span\{\partial_{u^{2}}\}$. We define $z_{0}=u^{2}$
, $z_{0,c}=\hat{u}^{1}$ and the distribution reads as
\[
D_{0}=\Span\{\partial_{z_{0}}\}\,.
\]
In the second step, due to the fact that $E_{1}$ is completely projectable,
case (a) applies. We just define $z_{1}=(\bar{x}_{1}^{1},z_{0,c})=(\bar{x}_{1}^{1},\hat{u}^{1})$,
and the distribution $D_{1}$ reads as
\[
D_{1}=\Span\{\partial_{z_{0}},\partial_{z_{1}}\}\,.
\]
 Similarly, the last distribution $E_{2}$ is also completely projectable
(cf. Lemma \ref{lem:EProj}) and thus we have again case (a). We just
define $z_{2}=(\bar{x}_{2}^{1},\bar{x}_{2}^{2})$, and the distribution
$D_{2}$ reads as
\[
D_{2}=\Span\{\partial_{z_{0}},\partial_{z_{1}},\partial_{z_{2}}\}\,.
\]
Consequently, with 
\[
\begin{aligned}z_{0} & =(u^{2})\,, & z_{1} & =(\bar{x}_{1}^{1},\hat{u}^{1})\,, & z_{2} & =(\bar{x}_{2}^{1},\bar{x}_{2}^{2})\end{aligned}
\]
the system has the structure of \eqref{eq:Expicit_Triangular_I_form}
and reads as
\begin{equation}
\begin{aligned}\hat{x}_{3}^{1,+} & =z_{2}^{1}+z_{2}^{2}\\
\hat{x}_{2}^{1,+} & =\hat{x}_{3}^{1}z_{1}^{2}+z_{1}^{1}\\
\hat{x}_{2}^{2,+} & =z_{1}^{2}\\
\hat{x}_{1}^{1,+} & =\hat{x}_{3}^{1}+z_{0}^{1}\,.
\end{aligned}
\label{eq:ex_z}
\end{equation}
As mentioned before, the subsystems of \eqref{eq:ex_z} may still
have redundant inputs. Indeed, because of $\dim(\hat{x}_{3})<\dim(z_{2})$,
the inputs $z_{2}^{1}$ and $z_{2}^{2}$ of the topmost subsystem
\[
\hat{x}_{3}^{1,+}=z_{2}^{1}+z_{2}^{2}
\]
are redundant. This redundancy can be eliminated by the final transformation
\begin{equation}
\bar{z}_{2}^{2}=z_{2}^{1}+z_{2}^{2}\,.\label{eq:z_state}
\end{equation}
Since $z_{2}^{2}$ represents a state variable, the equation \eqref{eq:z_state}
defines a state transformation and can be rewritten as 
\[
\bar{\hat{x}}_{2}^{2}=\hat{x}_{2}^{1}+\hat{x}_{2}^{2}\,.
\]
Collecting all transformations we performed so far, we obtain the
complete transformation
\begin{equation}
\begin{aligned}\hat{x}_{3}^{1} & =x^{1}(x^{3}+1) & \hat{x}_{1}^{1} & =x^{4}\\
\hat{x}_{2}^{1} & =x^{2}+3x^{4} & \hat{u}^{1} & =u^{1}+2u^{2}\\
\bar{\hat{x}}_{2}^{2} & =x^{3}+x^{2}+3x^{4} & u^{2} & =u^{2}\,,
\end{aligned}
\label{eq:example_1_block-1}
\end{equation}
which transforms the system \eqref{eq:example_automa} into the explicit
triangular normal form
\[
\begin{aligned}\hat{x}_{3}^{1,+} & =\bar{\hat{x}}_{2}^{2}\\
\hat{x}_{2}^{1,+} & =\hat{u}^{1}\hat{x}_{3}^{1}+\hat{x}_{1}^{1}\\
\bar{\hat{x}}_{2}^{2,+} & =\hat{u}^{1}\hat{x}_{3}^{1}+\hat{u}^{1}+\hat{x}_{1}^{1}\\
\hat{x}_{1}^{1,+} & =\hat{x}_{3}^{1}+u^{2}
\end{aligned}
\]
with the flat output $y=(\hat{x}_{3}^{1},\hat{x}_{2}^{1}).$ The parameterizing
map can now be constructed in a systematic way. From the first equation
we immediately get the parameterization of $\bar{\hat{x}}_{2}^{2}$.
Inserting this parameterization into the second and the third equation
yields the parameterization of $\hat{x}_{1}^{1}$ and $\hat{u}^{1}$.
Finally, inserting the parameterization of $\hat{x}_{1}^{1}$ into
the last equation yields the parameterization of $u^{2}$. With the
inverse of \eqref{eq:example_1_block-1}, the parameterization of
the original system variables $x$ and $u$ follows.

\section{Conclusion}

We have shown that every flat nonlinear discrete-time system with
two inputs can be transformed into a structurally flat explicit triangular
normal form. In contrast to the implicit triangular form discussed
in \cite{KolarDiwoldSchoberl:2019}, this normal form is a state representation.
The transformation is based on the sequences of distributions \eqref{eq:D_seq}
and \eqref{eq:Delta_seq}, that arise in the test for flatness introduced
in \cite{KolarDiwoldSchoberl:2019}. If it is possible to straighten
out both sequences of distributions simultaneously by state- and input
transformations, then the transformed system has the triangular structure
\eqref{eq:Expicit_Triangular_I_form}. For static feedback linearizable
systems, even in the multi-input case with $m>2$, this can always
be achieved by a state transformation. For flat systems that are not
static feedback linearizable, in contrast, there is no guarantee that
both sequences can be straightened out simultaneously, even if additionally
input transformations are permitted. However, for flat systems with
two inputs, straightening out \eqref{eq:D_seq} and \eqref{eq:Delta_seq}
by state- and input transformations is always possible. Thus, every
flat discrete-time system with two inputs can be transformed into
an explicit triangular form.

It is important to emphasize that for flat continuous-time systems
no comparable result exists. An obvious reason is that the explicit
triangular form allows to read off a flat output which depends only
on the state variables. In contrast to continuous-time systems, it
is shown in \cite{KolarSchoberlDiwold:2019} that for flat discrete-time
systems such a flat output always exists.
\appendix

\subsection{Proof of Lemma \ref{lem:An-involutive-distribution}}

Due to the involutivity, all pairwise Lie Brackets must be contained
in $D$, i.e. 
\[
[\partial_{\zcord^{l}},\partial_{\zcord^{i}}+\alpha(\zcord)\partial_{\zcord^{j}}]\subset D\,,\hphantom{aa}l=1,\dots,k-1\,.
\]
Because of the special structure of the basis of $D$, this implies
that all pairwise Lie brackets vanish identically. Consequently, the
coefficient $\alpha$ meets $\partial_{\zcord^{l}}\alpha=0$ for $l=1,\dots,k-1$,
i.e., $\alpha$ is independent of $\zcord^{1},\dots,\zcord^{k-1}$.
Next, the flow $\phi_{t}(\zcord_{0})$ of the vector field $\partial_{\zcord^{i}}+\alpha(\zcord^{k},\dots,\zcord^{n})\partial_{\zcord^{j}}$
is of the form
\begin{align*}
\zcord^{i}(t,\zcord_{0}) & =t+\zcord_{0}^{i}\\
\zcord^{j}(t,\zcord_{0}) & =\phi_{t}^{j}(\zcord_{0}^{k},\dots,\zcord_{0}^{n})\,,
\end{align*}
i.e., it only affects the coordinates $\zcord^{i}$ and $\zcord^{j}$.
According to the flow-box theorem, by setting $t=\zcord^{i}$, $\zcord_{0}^{i}=0$,
$\zcord_{0}^{j}=\hat{\zcord}^{j}$ and $\zcord_{0}^{l}=\zcord^{l}$
for $l=k,\dots,n$ with $l\neq i,j$ on the right hand side, we obtain
a coordinate transformation which transforms the above vector field
into the form $\partial_{\zcord^{i}}$. In fact, only $\zcord^{j}$
is replaced by the transformed coordinate $\hat{\zcord}^{j}$, and
all other coordinates remain unchanged. In these coordinates, the
distribution $D$ reads as \eqref{eq:straight_Z}. The inverse coordinate
transformation is of the form \eqref{eq:z_transf}.

\subsection{Proof of Theorem \ref{thm:Triangular_Form}}

First, we show that the variables $z_{0},\dots,z_{\bar{k}-1}$ contain
all inputs and states except $\hat{x}_{\bar{k}}$. Since the distributions
\eqref{eq:straight_Delta} are straightened out, according to \eqref{eq:straight_E}
the distribution $E_{\bar{k}-1}$ reads as 
\[
E_{\bar{k}-1}=\Span\{\partial_{u},\partial_{\hat{x}_{1}},\dots,\partial_{\hat{x}_{\bar{k}-1}}\}\,.
\]
Lemma \ref{lem:EProj} guarantees that $E_{\bar{k}-1}$ is completely
projectable, and thus it coincides with the distribution 
\[
D_{\bar{k}-1}=\Span\{\partial_{z_{0}},\dots,\partial_{z_{\bar{k}-1}}\}\,.
\]
Thus, the variables $z_{0},\dots,z_{\bar{k}-1}$ contain all inputs
and states except $\hat{x}_{\bar{k}}$. The property \eqref{eq:x_sub_z}
is a consequence of $D_{k-1}\subset E_{k-1}$, i.e. 
\begin{equation}
\Span\{\partial_{z_{0}},\dots,\partial_{z_{k-1}}\}\subset\Span\{\partial_{u},\partial_{\hat{x}_{1}},\dots,\partial_{\hat{x}_{k-1}}\}\label{eq:D_kqer_1=00003DE_kqer_1}
\end{equation}
and $D_{\bar{k}-1}=E_{\bar{k}-1}$, i.e.
\begin{equation}
\Span\{\partial_{z_{0}},\dots,\partial_{z_{\bar{k}-1}}\}=\Span\{\partial_{u},\partial_{\hat{x}_{1}},\dots,\partial_{\hat{x}_{\bar{k}-1}}\}\,.\label{eq:D=00003DE}
\end{equation}
Because of \eqref{eq:D_kqer_1=00003DE_kqer_1}, the variables $\hat{x}_{k}$
cannot be contained in $(z_{0},\dots,z_{k-1})$. However, according
to \eqref{eq:D=00003DE}, they must be contained in $(z_{k},\dots,z_{\bar{k}-1})$.

The triangular structure of \eqref{eq:Expicit_Triangular_I_form}
is a consequence of 
\begin{equation}
f_{*}(D_{k-1})=\Delta_{k}=\Span\{\partial_{\hat{x}_{1}^{+}},\dots,\partial_{\hat{x}_{k}^{+}}\}\,,\hphantom{aa}k=1,\ldots,\bar{k}\,.\label{eq:f*D}
\end{equation}
For $k=0$, from \eqref{eq:f*D} and $D_{0}=\Span\{\partial_{z_{0}}\}$
we get $\partial_{z_{0}}f_{i}=0$ for $i=2,\dots,\bar{k}$, i.e.
\begin{align*}
\hat{x}_{\bar{k}}^{+} & =f_{\bar{k}}(\hat{x}_{\bar{k}},z_{\bar{k}-1},\dots,z_{1})\\
 & \vdots\\
\hat{x}_{2}^{+} & =f_{2}(\hat{x}_{\bar{k}},z_{\bar{k}-1},\dots,z_{1})\\
\hat{x}_{1}^{+} & =f_{1}(\hat{x}_{\text{\ensuremath{\bar{k}}}},z_{\bar{k}-1},\dots,z_{0})\,.
\end{align*}
Furthermore, because of $\dim(\Delta_{1})=\dim(\hat{x}_{1})$, the
rank condition $\text{rank}(\partial_{z_{0}}f_{1})=\dim(\hat{x}_{1})$
follows. Next, for $k=1$, from \eqref{eq:f*D} and $D_{1}=\Span\{\partial_{z_{0}},\partial_{z_{1}}\}$,
we get $\partial_{z_{1}}f_{i}=0$ for $i=3,\dots,\bar{k}$, i.e.
\begin{align*}
\hat{x}_{\bar{k}}^{+} & =f_{\bar{k}}(\hat{x}_{\bar{k}},z_{\bar{k}-1},\dots,z_{2})\\
 & \vdots\\
\hat{x}_{3}^{+} & =f_{3}(\hat{x}_{\bar{k}},z_{\bar{k}-1},\dots,z_{2})\\
\hat{x}_{2}^{+} & =f_{2}(\hat{x}_{\bar{k}},z_{\bar{k}-1},\dots,z_{1})\\
\hat{x}_{1}^{+} & =f_{1}(\hat{x}_{\text{\ensuremath{\bar{k}}}},z_{\bar{k}-1},\dots,z_{0})\,.
\end{align*}
Again, because of $\dim(\Delta_{2})=\dim(\hat{x}_{1})+\dim(\hat{x}_{2})$,
the rank condition $\text{rank}(\partial_{z_{1}}f_{2})=\dim(\hat{x}_{2})$
follows. Repeating this argumentation shows that the system has the
triangular structure \eqref{eq:Expicit_Triangular_I_form} and meets
the rank conditions \eqref{eq:Rankcond}.

\bibliographystyle{IEEEtran}
\bibliography{IEEEabrv,Bibliography_Johannes_April_2020}

\end{document}